\documentclass[12pt]{amsart}

\usepackage{amsaddr,amsmath,amsthm,amssymb,hyperref}

\newtheorem{theorem}{Theorem}
\newtheorem{proposition}[theorem]{Proposition}
\newtheorem{lemma}[theorem]{Lemma}

\newtheorem{question}[theorem]{Question}
\theoremstyle{definition}
\newtheorem{definition}[theorem]{Definition}

\theoremstyle{remark}
\newtheorem{remark}[theorem]{Remark}

\renewcommand{\rm}[1]{\mathrm{#1}}
\newcommand{\bb}[1]{\mathbb{#1}}
\newcommand{\fk}[1]{\mathfrak{#1}}
\newcommand{\bd}[1]{\mathbf{#1}}
\newcommand{\cl}[1]{\mathcal{#1}}
\renewcommand{\a}{\alpha}
\renewcommand{\b}{\beta}
\newcommand{\g}{\gamma}

\renewcommand{\d}{\delta}
\newcommand{\e}{\epsilon}

\newcommand{\om}{\omega}
\newcommand{\Om}{\Omega}
\newcommand{\w}{\wedge}

\newcommand{\tr}{\mathrm{tr}}

\newcommand{\pt}[2]{\frac{\partial #1}{\partial #2}}

\newcommand{\vol}{\mathrm{vol}}
\newcommand{\inv}{\mathrm{inv}}

\begin{document}

\title{$SO(3)$-invariant $G_2$-cobordisms}
\author{Ryohei Chihara}
\address{Graduate School of Mathematical Sciences, University of Tokyo, 3-8-1 Komaba, Meguro-ku, Tokyo 153-8914, Japan}
\email{rchihara@ms.u-tokyo.ac.jp}

\maketitle

\begin{abstract}
We study a bordism relation for stable 3-forms on a 6-manifold, which is a binary relation on the set of closed $SL(3;\bb{C})$-structures on a 6-manifold via closed $G_2$-structures. Under $SO(3)$-symmetry and a co-associative condition the relation is reduced to  a relation for geometric structures on a 3-manifold. Under these conditions we prove that the bordism relation is irreflexive and that the relation induces a more rigid one.
\end{abstract}

\section{Introduction}
\label{sec:intro}

The concept of $G_2$-cobordisms between 6-manifolds with $SL(3;\bb{C})$-structures, which was introduced by S. Donaldson (\cite{D}, Section 4), is significant in the context of boundary value problems for metrics with holonomy $G_2$. The bordism relation is defined as follows. Let $X$ be a spin 6-manifold, and let $\psi_1,\psi_2$  be closed definite 3-forms that give $SL(3;\bb{C})$-structures on $X$. The forms $\psi_1$ and $\psi_2$ are {\it $G_2$-cobordant} if there exists a closed definite 3-form $\phi$ defining a $G_2$-structure on $X\times[t_1,t_2]$ compatible with the orientation and satisfying $\phi|_{X\times\{t_i\}}=\psi_i$ for $i=1,2$. This gives a binary relation $\psi_1 \prec \psi_2$ on the set $\cl{C}(X)$ of closed definite 3-forms on $X$.

Moreover we define a more rigid  relation $\psi_1\ll\psi_2$ on $\cl{C}(X)$ by the existence of such a closed $G_2$-structure $\phi$ that also induces a Riemannian metric on $X\times[t_1,t_2]$ such that the product is orthogonal.

 Although Donaldson pointed out that the relation $\prec$ is transitive (\cite{D}, Proposition 2), many basic properties of the relations are still open. In this paper, we are interested in the following questions:
 
 \begin{question}\label{conj:1} Let $X$ be a closed spin 6-manifold. When is the binary relation $\prec$ on $\cl{C}$(X) irreflexive? Here a binary relation is {\it irreflexive} if it satisfies $\psi\nprec\psi$ for all $\psi\in \cl{C}(X)$.
\end{question}

Note that if $X$ admits a Calabi-Yau structure $(\om, \psi)$ then we can easily see that the 3-form satisfies $\psi \prec \psi$ using the symplectic form $\om$. So the non-existence of Calabi-Yau structures is a necessary condition for the relation to be irreflexive.

\begin{question}\label{conj:2}
When does the relation $\prec$ induce the above more rigid relation $\ll$?
\end{question}

For a Calabi-Yau structure $(\om,\psi)$ we have $\psi \ll \psi$. We expect that this question is related to sub-Riemannian geometry of diffeomorphism group action to $\cl{C}(X)$.\\

In this paper, we consider the relation under $SO(3)$-symmetry and a co-associative condition, and prove several results related to the above questions. More precisely, we study $G_2$-cobordisms defined by $SO(3)$-invariant co-associative fibrations $M\times SO(3) \times[t_1,t_2]$ over a base 3-manifold $M$. These give a binary relation on the set of $SO(3)$-invariant closed definite 3-forms on $M\times SO(3)$ vanishing along the fibers $SO(3)$. We describe the relation explicitly (Lemma \ref{lem:der} and Theorem \ref{thm:1}), and prove that the relation is irreflexive (Theorem \ref{thm:2}). Moreover, we prove that under these conditions the relation induces the more rigid one (Theorem \ref{thm:3}). 

\section{Preliminaries}
\label{sec:pre}

In this section we recall the notions of definite 3-forms and $G_2$-cobordisms. The former was studied in detail in \cite{H, H2} and the latter introduced in \cite{D}. 

\subsection{Definite 3-forms in dimension 6 and 7}

Let $V$ be a real vector space of dimension 6 and $V^*$ its dual. Take a basis $\{v^1,w^1,v^2,w^2,v^3,w^3\}$ of $V^*$.  Let us set a normal form
\begin{align*}
\psi_0 &=\rm{Im}\{(v^1+\sqrt{-1}w^1)\w(v^2+\sqrt{-1}w^2)\w(v^3+\sqrt{-1}w^3)\}\\
&=-w^{123}+w^1v^{23}+w^2v^{31}+w^3v^{12}
\end{align*}
where e.g., $v^{12}=v^1\w v^2$. A 3-form 
$\psi$ on $V$ is called {\it definite} if $\psi$ is contained in the orbit $GL(V)\cdot\psi_0 \subset \w^3 V^*$. This condition is equivalent to saying that $\iota(v)\psi \in \w^2V^*$ has rank 4 for any non-zero $v\in V$, where $\iota$ denotes inner product (see \cite{D}, Section 2.1). The definiteness of a 3-form is an open condition. 

Let $X$ be an oriented 6-manifold. If $X$ has an $SL(3;\bb{C})$-structure, then  $X$ is spin. A 3-form $\psi\in\Om^3(X)$ is called {\it definite}
 if $\psi_{x}\in \w^3T^*_xX$ is definite at each point $x\in X$. We know that the group $\{g\in GL^{+}(V)\mid g^{*}\psi_0=\psi_0 \}$ coincides with $SL(3; \bb{C})\subset GL(6;\bb{R})$. Thus a definite 3-form $\psi$ naturally gives an $SL(3;\bb{C})$-structure on $X$. In particular, $\psi$ gives an almost complex structure $J_{\psi}$ on $X$.
 
 Let $W$ be a 7-dimensional real vector space and $W^*$ its dual. Take a basis $\{v^0,v^1,w^1,v^2,w^2,v^3,w^3\}$, and set a normal form
 \begin{align*}
     \phi_0 &= (v^1w^1+v^2w^2+v^3w^3)v^0+\psi_0 \\
  &= -w^{123}+w^1(v^{01}+v^{23})+w^2(v^{02}+v^{31})+w^3(v^{03}+v^{12}).
 \end{align*}
 A 3-form $\phi$ on $W$ is called {\it definite} if $\phi$ is contained in the orbit $GL(W)\cdot\phi_0 \subset \w^3W^*$. This condition is equivalent to saying that $\iota(w)\phi$ has rank 6 for any non-zero $w\in W$ (see \cite{D}, Section 2.1). Moreover, the restriction of a definite 3-form to any 6-dimensional subspace is also definite. 
 
 Let $Y$ be a 7-manifold. If $Y$ has a $G_2$-structure, then $Y$ is spin. A 3-form $\phi \in \Om^3(Y)$ is called  {\it definite} if $\phi_y\in \w^3T^*_yY$ is definite for each $y\in Y$. The isotropy group of $\phi_0$ is known to coincide with $G_2 \subset SO(7)$. Thus a definite 3-form $\phi$ gives a $G_2$-structure, which induces an orientation and a Riemannian metric on $Y$. Let us denote by $\rm{vol}(\phi)$ the volume form on $Y$ induced by $\phi$. See e.g., \cite{D, H, H2, J} for more details on $SL(3;\bb{C})$- and $G_2$-structures.
 
 \subsection{$G_2$-cobordisms}
 
Let $X$ be a spin 6-manifold. Define 
\begin{align*}
    \cl{C}(X):=\{\psi \in \Om^3(X) \mid \psi \text{ is closed and definite}\}.
\end{align*}
This set is a subset of the $SL(3;\bb{C})$-structures compatible with the orientation on $X$. Let us define a binary relation on $\cl{C}(X)$ using closed $G_2$-structures on $X\times[t_1,t_2]$, where $t_1 < t_2$.

\begin{definition}\label{def:binary}
Let $\psi_1$, $\psi_2\in\cl{C}(X)$. Then 
$\psi_1 \prec \psi_2$ if there exists a closed definite 3-form $\phi$ on $X\times[t_1,t_2]$ satisfying the conditions: $\vol(\phi)>0$ for the product orientation, $\phi|_{X\times\{t_1\}}=\psi_1$ and $\phi|_{X\times\{t_2\}}=\psi_2$.
\end{definition}

This binary relation is linked to non-degenerate 2-forms on $X$ by the following elementary proposition.

\begin{proposition}[\cite{D}, the second paragraph in Section 4]\label{prop:rewrite} 
Let $\psi_1,\psi_2\in \cl{C}(X).$
Then $\psi_1\prec\psi_2$ if and only if there exist $\om_t\in\Om^2(X)$ and $\psi_t\in\Om^3(X)$ parameterized by $t\in [t_1,t_2]$ satisfying the following conditions:
\begin{enumerate}
\item $\psi_{t_1}=\psi_1$, $\psi_{t_2}=\psi_2$, and $\psi_t$ is definite for each $t$,

\item $\partial \psi_t/\partial t = d\om_t$ for each $t$, 

\item the (1,1) part of $\om_t$ is positive with respect to the almost complex structure on $X$ induced by each $\psi_t$.
\end{enumerate}
\end{proposition} 

\begin{proof}
Suppose that $\psi_1\prec\psi_2$. Then we have a closed definite 3-form $\phi$ on $X\times[t_1,t_2]$ satisfying the conditions in Definition \ref{def:binary}. Setting $\om_t:=\iota(\partial/\partial t)\phi|_{X\times\{t\}}$ and $\psi_t:=\phi|_{X\times\{t\}}$, we obtain the desired forms. Conversely, suppose that we have $\om_t$ and $\psi_t$ satisfying the conditions in Proposition \ref{prop:rewrite}. Then $\phi=\om_t\w dt+\psi_t$ is the desired closed definite 3-form on $X\times[t_1,t_2]$, which gives $\psi_1 \prec \psi_2$.
\end{proof}
 
 \begin{remark}
  By Proposition \ref{prop:rewrite}, we can easily see that $\psi \prec \psi$ by some one-parameter family $(\om_t,\psi_t)$ with $\psi_t \equiv \psi$ if and only if $\psi$ has a symplectic form $\om$ that has the positive $(1,1)$ part with respect to $J_{\psi}$. Moreover we can give a non-constant path giving $\psi \prec \psi$ by perturbing
 $\om_t$ to $\om_t +\a_t$. Here $\a_t$ is a family of small 2-forms satisfying $\int_{[t_1,t_2]}\a_tdt=0$ pointwise.
\end{remark}

Let us define a more rigid version of the above relation.
 
\begin{definition}\label{def:stronger}
Let $\psi_1,\psi_2\in\cl{C}(X)$. Define  $\psi_1\ll \psi_2$ by that  there exists a closed definite 3-form $\phi$ on $X\times[t_1,t_2]$ satisfying the conditions in Definition \ref{def:binary} and also $\iota(\partial/\partial t)\phi\w\phi|_{X\times\{t\}}=0$ for each $t$.
\end{definition}

The last condition is equivalent to saying that the 2-form  $\iota(\partial/\partial t)\phi|_{X\times\{t\}}$ is a $(1,1)$ form on $X$ with respect to the almost complex structure induced by each $\phi|_{X\times\{t\}}$. This condition is also equivalent to that the product $X \times [t_1,t_2]$ is orthogonal for the metric defined by $\phi$.

\section{$SO(3)$-invariant structures}
\label{sec:main}

In this section we prove several results related to Questions \ref{conj:1} and \ref{conj:2} under $SO(3)$-symmetry and a co-associative condition.

\subsection{Set-up}

Let $M$ be an oriented 3-manifold. Fix a basis $\{X_1,X_2,X_3\}$ of $\fk{so}(3)$ satisfying $[X_i,X_j]=\e_{ijk}X_k$ for $i,j=1,2,3,$ where $\e_{ijk}$ is the 3-rd order Levi-Civita symbol. An orientation on $SO(3)$ is given by $X_1\w X_2\w X_3$, and that on $M\times SO(3)$ by direct product.
Let us define  $\psi\in \cl{C}^{\inv}(M)$ by the following conditions:
\begin{enumerate}
    \item $\psi\in\cl{C}(M\times SO(3))$,
    \item $\psi$ is $SO(3)$-invariant under the right action,
    \item $\psi$ vanishes along each fiber $SO(3)$.
\end{enumerate}
Note that $\cl{C}(M) \subset \cl{C}(M \times SO(3))$. Refinements of the relations $\prec$ and $\ll$ in Definition \ref{def:binary} and \ref{def:stronger} are defined as follows.

\begin{definition}\label{def:invbinary}
Let $\psi_1,\psi_2\in\cl{C}^{\inv}(M)$. Then $\psi_1\prec^{\inv}\psi_2$ if there exists a closed definite 3-form $\phi$ on $M\times SO(3) \times[t_1,t_2]$ satisfying the conditions in Definition \ref{def:binary} and also that  $\phi$ is $SO(3)$-invariant and vanishes along each fiber  $SO(3)\times[t_1,t_2]$.
\end{definition}

In Definition \ref{def:invbinary} each fiber $SO(3)\times[t_1,t_2]$ is an almost co-associative submanifold of the 7-manifold $M\times SO(3)\times[t_1,t_2]$ with the $G_2$-structure $\phi$.

\begin{definition}
Let $\psi_1,\psi_2\in\cl{C}^{\inv}(M)$. Then $\psi_1\ll^{\inv}\psi_2$ if there exists a closed definite 3-form $\phi$ satisfying the conditions in Definition \ref{def:invbinary} and also that  $\iota(\partial/\partial t)\phi\w\phi|_{M\times SO(3) \times\{t\}}=0$ for each $t$.
\end{definition}

Clearly $\psi_1\prec^{\inv}\psi_2$ implies $\psi_1\prec\psi_2$, and $\ll^{\inv}$ implies $\ll$.

\subsection{Lemmas}

Let us prove lemmas used in the proofs of main results. Let $C_{+}^{\infty}(M)$ and $\cl{A}$ be the set of positive functions on $M$ and the set of connection 1-forms on the trivial bundle $M\times  SO(3)$, respectively. A tensorial (or horizontal) $\fk{so}(3)$-valued 1-form $e$ on $M\times  SO(3)$ is called a {\it solder 1-form} if $e=e^iX_i$ satisfies $e^{123} \neq 0$ at each $u \in M\times  SO(3)$ for the basis $\{X_1,X_2,X_3\} \subset \fk{so}(3)$. By this basis, we write $a=a^iX_i \in \cl{A}$ and often follow Einstein's convention.

\begin{lemma}\label{lem:dec} 
Let $\psi$ be an $SO(3)$-invariant definite 3-form on $M\times SO(3)$ vanishing along each fiber $SO(3)$. Then there exists a unique triple of  $f\in C^{\infty}_{+}(M)$, $a\in \cl{A}$, and a solder 1-form $e$ such that 
\begin{align*}
    \psi=-fe^{123}+e^1a^{23}+e^2a^{31}+e^3a^{12}.
\end{align*}
\end{lemma}

\begin{proof}
We can easily see that there exists an $SO(3)$-invariant positive $(1,1)$ form $\om$ on $M\times SO(3)$ vanishing along each fiber. In fact, let $V_1,V_2,V_3$ be $SO(3)$-invariant linear-independent vector fields along the fibers, and let $J_{\psi}$ be the almost complex structure induced by $\psi$. Put $W_i=J_{\psi}V_i$ for $i=1,2,3.$ Let $\{V^i,W^i\mid i=1,2,3\}$ be the dual 1-forms of $\{V_i,W_i\mid i=1,2,3\}$. Then we can see that $\sum_{i=1}^3V^i\w W^i$ is the desired 2-form. Hence, by (\cite{C}, Theorem 3.5), there exists a triple $(f,a,e)$ satisfying the desired conditions. We can easily see the uniqueness.
\end{proof}

\begin{remark} We remark that $d\psi = 0$ if and only if $a$ is the Levi-Civita connection satisfying $de + [a\w e]=0$ and moreover that such a closed $SO(3)$-invariant 3-form is thus determined by $(f, e)$, and up to action of the gauge group by $(f,\g)$ for $\g$ the metric on $M$ defined by $e$. (See \cite{C2}, Subsection 1.2 for details.)
\end{remark}

\begin{lemma}\label{lem:dec2}
Let $\psi$ be a 3-form as in Lemma \ref{lem:dec}, and $\om$ an $SO(3)$-invariant 2-form  on $M\times SO(3)$ vanishing along all fibers and horizontal planes of the connection $a$. Then there exists a unique $SO(3)$-equivariant $M(3;\bb{R})$-valued function $K_{ij}$ on $M\times SO(3)$ such that 
\begin{align*}
    \om= \sum_{i,j=1,2,3}K_{ij}a^i\w e^j.
\end{align*}
Here $C \in SO(3)$ acts on $D \in M(3;\bb{R})$ by $CDC^{-1}$.
\end{lemma}

\begin{proof}
Using the basis $\{X_1,X_2,X_3\}$ of $\fk{so}(3)$, let us define $K_{ij}$ by $\iota(X^*_i)\om = \sum_{j=1}^3K_{ij}e^j$ for $i=1,2,3,$ where $A^*$ denotes the infinitesimal vector field of $A\in \fk{so}(3)$. Then $K_{ij}$ satisfies the conditions in Lemma \ref{lem:dec2}.
\end{proof}

\begin{remark}
We can easily see that $\om$ is (1,1) with respect to the almost complex structure induced by $\psi$ if and only if $K_{ij}$ is symmetric.
\end{remark}

Let $(\om_t,\psi_t)$ be a one-parameter family of 2- and 3-forms satisfying the conditions in Lemma \ref{lem:dec} and \ref{lem:dec2} and parameterized by an  interval $[t_1,t_2]$. Then we have the family $(f_t,a_t,e_t,(K_{ij})_t)$ corresponding to $(\om_t,\psi_t)$ by Lemma \ref{lem:dec} and \ref{lem:dec2}. The solder form $e$ allows us to identify $a$ with a connection on $M$, and hence make sense of the condition that $a$ is the Levi-Civita connection. We write the curvature form $da + (1/2)[a\w a]$ as $G_{ij}\hat{e}^iX_j$, where $\hat{e}^i =(1/2)\e_{ijk}e^{ij}$.

\begin{lemma}\label{lem:der} 
The 3-form $\phi=\om_t\w dt+\psi_t$ on $M\times SO(3)\times[t_1,t_2]$ is closed, definite and $\vol(\phi)>0$ if and only if 
$(f_t,a_t,e_t,(K_{ij})_t)$ satisfies the following conditions:
\begin{enumerate}
    \item $(1/2)\{(K_{ij})_t+(K_{ji})_t\}$ is positive-definite for each $t$,
    \item $a_t$ is the Levi-Civita connection for each $e_t$,
    \item $\partial e^i / \partial t= \sum_{j=1}^3 (K_{ij})_te_t^j$ for $i=1,2,3$ and for each $t$,
    \item $\partial f / \partial t=-f_t\tr{K_t}-\tr {(K_tG_t)}$ for each $t$, where $G_t$ is the curvature of $a_t$.
\end{enumerate}
\end{lemma}

\begin{proof}
We can prove this lemma by direct computation as the proof of (\cite{C}, Proposition 6.8). First, we can easily see that $\phi$ is definite and $\vol(\phi)>0$ if and only if the first condition holds. Thus, all we have to do is to compute $d\phi$. We have
\begin{align*}
    d\phi &= d(\om_t\w dt + \psi_t)\\
    &= (d\om_t - \partial \psi /\partial t)\w dt + d\psi_t, 
\end{align*}
thereby $d\phi=0$ is equivalent to $d\psi_t=0$ and $\partial \psi/\partial t=d\om_t$. By (\cite{C}, Corollary 6.5 (1)), we see that $d\psi_t=0$ if and only if $a_t$ is Levi-Civita for $e_t$. From now on, suppose that $d\psi_t=0.$ Put $\partial e^i /\partial t = P_{ij}e_t^j$ and $\partial a^i /\partial t= Q_{ij}e^j_t$ for $i=1,2,3$. We have 
\begin{align*}
    &\pt{\psi}{t}=-\{\pt{f}{t}+f\tr{P}\}e^{123} +\{\d_{ij}\tr{Q}-Q_{ji}\}a^i\hat{e}^{j} + P_{ij}\hat{a}^ie^j, \\
    &d\om=\tr{(KG)}e^{123}+\e_{j\a\b}K_{i\a;\b}a^i\hat{e}^j + K_{ij}\hat{a}^ie^j,
\end{align*}
where we omit the subscript $t$, denote by $K_{i\a;\b}e^{\b}$ the covariant derivative of $(K_{ij})_t$ for the connection $a_t$,  $\hat{e}^i=(1/2)\e_{ijk}e^{ij}$ and $\hat{a}^i=(1/2)\e_{ijk}a^{jk}$. Also $\delta_{ij}$ is Kronecker's delta. As seen in (\cite{C2}, Lemma 10), the Levi-Civita condition $de_t+[a_t\w e_t]=0$ implies $\delta_{ij}\tr{Q}-Q_{ji}=\e_{j\a\b}K_{i\a;\b}a^i$ for $i,j=1,2,3$. Hence, by comparing the equations above, we obtain the conditions in Lemma \ref{lem:der}.
\end{proof}

\subsection{Results}\label{subsec:res}

Let $\bd{M}$ be the set of all Riemannian metrics on $M$. Using Lemma \ref{lem:dec}, we can define the projection $\pi: \cl{C}^{\inv}(M)\to C^{\infty}_{+}(M)\times \bd{M}$ by $\pi(\psi)=(f,\g)$ for each $\psi \in \cl{C}^{\inv}(M)$, where we take $f$ and $e$ as in Lemma \ref{lem:dec}, and $\g$ is the Riemannian metric on $M$ naturally induced by the solder 1-form $e$. 
Let us define a binary relation on $C^{\infty}_{+}(M)\times\mathbf{M}$. 
Let $(f_1,\g_1),(f_2,\g_2)\in C^{\infty}_{+}(M)\times \bd{M}$. Then $(f_1,\g_1)\prec(f_2,\g_2)$ if there exists a one-parameter family $(f_t,\g_t) \in C^{\infty}_{+}(M)\times \bd{M}$ parameterized by $t\in[t_1,t_2]$ satisfying the following conditions:
\begin{enumerate}
    \item $(f_{t_1},\g_{t_1})=(f_1,\g_1)$ and $(f_{t_2},\g_{t_2})=(f_2,\g_2)$,
    \item $\partial \g/ \partial t$ is positive-definite covariant symmetric tensor for each $t$,
    \item $\partial f / \partial t= - A(f,\g)^{ij}\g_{ij}$ for each $t$. Here $A(f,\g)^{ij}=1/2(G_t^{ij}+f\g_t^{ij})$ and $(G_{ij})_t$ is the Einstein tensor for each metric $\g_t$.
\end{enumerate}

\begin{theorem}\label{thm:1}
Let $\psi_1,\psi_2\in\cl{C}^{\inv}(M)$. Then $\psi_1\prec^{\inv}\psi_2$ implies $\pi(\psi_1)\prec\pi(\psi_2)$.
\end{theorem}

\begin{proof}
Suppose that $\psi_1\prec^{\inv}\psi_2.$ If we write a 3-form $\phi$ satisfying the conditions in Definition \ref{def:invbinary} as $\om_t\w dt + \psi_t$, then applying Lemma \ref{lem:dec} to $\psi_t$ gives the decomposition $(f_t,a_t)$, by the co-associative condition in Definition \ref{def:invbinary} we can apply Lemma \ref{lem:dec2} to $\om_t$ to obtain the decomposition $(f_t,a_t,e_t,K_t)$ of $(\om_t,\psi_t)$, and then apply Lemma \ref{lem:der} to deduce $\pi(\psi_1)\prec\pi(\psi_2)$. 
\end{proof}

\begin{theorem} \label{thm:2}
The relation $\prec^{\inv}$ on $\cl{C}^{\inv}(M)$ is irreflexive, that is, $\psi \nprec^{\inv} \psi$ for any $\psi\in \cl{C}^{\inv}(M)$.
\end{theorem}

\begin{proof} From Lemma \ref{lem:der}, we can see that
if $\psi_1\prec^{\inv}\psi_2$ then we have $\partial (e^{1}_t\w e^{2}_t\w e^3_{t}) /\partial t = \tr{K_t}e^{1}_t\w e^{2}_t\w e^3_{t}$ and $\tr{K_t}>0$. Hence $\psi_1\neq\psi_2$.
\end{proof}

\begin{remark}
If an oriented 3-manifold $M$ is closed,  then $M\times SO(3)$ admits no symplectic structure. We do not know whether the non-existence of symplectic structures induces irreflexivility of $\prec$ in general.
\end{remark}

Let $\cl{G}$ be the gauge group of the trivial principal bundle $M\times SO(3)$.

\begin{proposition}\label{prop:lift}
Let $\psi_1,\psi_2\in\cl{C}^{\inv}(M)$. If $\pi(\psi_1)\prec\pi(\psi_2)$, then there exists $\tau\in \cl{G}$ such that $\psi_1\ll^{\inv}\tau^*\psi_2$.
\end{proposition}

\begin{proof}
Suppose that $\pi(\psi_1)\prec\pi(\psi_2)$. Then we have a one-parameter family $(f_t,\g_t)$ parameterized by $[t_1,t_2]$ that gives the relation $\pi(\psi_1)\prec\pi(\psi_2)$. 
We can lift this curve $(f_t,\g_t)$ by the natural connection in the principal $\cl{G}$-bundle $\cl{M}\to \mathbf{M}$ used in (\cite{C2}, Section 1.1), where $\cl{M}$ denotes a connected component of the space of solder 1-forms on $M\times SO(3)$. Note that the set of solder 1-forms has two connected component isomorphic to each other. Here, we choose the one containing the solder 1-form given by $\psi_1$ as in Lemma \ref{lem:dec}.
 The connection in $\cl{M}\to \mathbf{M}$ is defined by the decomposition $M(3;\bb{R})=\rm{Sym}(3;\bb{R})\oplus\rm{Ant}(3;\bb{R})$ at each $T_{e}\cl{M}$, where $\rm{Sym}(3;\bb{R})$ and $\rm{Ant}(3;\bb{R})$ are the symmetric and anti-symmetric matrices. Let $(f_t,e_t)$ be a horizontally lifted curve in $C^{\infty}_{+}\times\cl{M}$, and define $(T_{ij})_t$ by $\partial e^i /\partial t = (T_{ij})_te^j_t$ for each $t$. Here, by the definition of the connection in $\cl{M}$, $(T_{ij})_t$ is an $SO(3)$-equivariant $\rm{Sym}(3)$-valued funcion on $M\times SO(3)$ for each $t$. Then a pair of $\psi_t=-f_te^{123}_t+e^1_ta^{23}_t+e^2_ta^{31}_t+e^3_ta^{12}_t$ and $\om_t=(T_{ij})_ta^i_t\w e^j_t$ gives $\psi_1\ll^{\inv} \psi_{t_2}$, where $a_t$ is the Levi-Civita connection for each $e_t$. Since $\pi(\psi_{t_2})=\psi_2$, there exists a unique $\tau\in\cl{G}$ such that $\psi_{t_2}=\tau^*\psi_2$.
\end{proof}

By Theorem \ref{thm:1} and Proposition \ref{prop:lift} combined, we have 

\begin{theorem}\label{thm:3}
Let $\psi_1,\psi_2\in \cl{C}^{\inv}(M)$. If $\psi_1\prec^{\inv}\psi_2$,  then there exists $\tau\in \cl{G}$ such that $\psi_1\ll^{\inv}\tau^*\psi_2$.
\end{theorem}

Theorem \ref{thm:3} is related to Question \ref{conj:2}. Moreover, considering how close we can take $\tau^*\psi_2$ to $\psi_2$ in Theorem  \ref{thm:3} seems to be related to the isoholonomic problem (see e.g. \cite{M, M2}), which is a typical problem of sub-Riemannian geometry, for the infinite-dimensional principal $\cl{G}$-bundle $\cl{M} \to \mathbf{M}$ over the space $\bd{M}$ of all Riemannian metrics on the 3-manifold $M$.

\section*{Acknowledgments}

The author thanks N.\ Kawazumi for many comments on the manuscript, R.\ Goto, H.\ Kasuya and H.\ Konno for discussing with him, and the referee for constructive comments. He is grateful to participants and organizers of East Asian Symplectic Conference 2019 for giving many comments on his talk. 
This work was supported by the Leading graduate course for Frontiers of Mathematical Sciences and Physics.

\bibliographystyle{amsplain}
\bibliography{paper3}

\end{document}